\documentclass[a4paper]{amsart}

\usepackage[english]{babel}
\usepackage[utf8]{inputenc}
\usepackage[T1]{fontenc}
\usepackage{amsmath}
\usepackage{amssymb}
\usepackage{multirow}
\usepackage{alltt}
\usepackage{graphicx}
\usepackage{bbm}
\usepackage{amsthm}
\usepackage{mathrsfs}
\usepackage{textcomp}
\usepackage{calc}
\usepackage{fullpage}
\usepackage[abbrev]{amsrefs}
\usepackage{esint}

\usepackage[usenames,dvipsnames]{color}
\definecolor{darkblue}{RGB}{0,0,170}
\definecolor{brickred}{RGB}{200,0,0}
\usepackage[pdfauthor={Nicola Abatangelo}, hyperindex, linktocpage]{hyperref}
\hypersetup{colorlinks=true, linkcolor=darkblue, citecolor=brickred}

\newcommand{\R}{\mathbb{R}}
\newcommand{\N}{\mathbb{N}}

\newcommand{\dist}{\mathrm{dist}}

\newcommand{\diam}{\mathrm{diam}}

\newcommand{\cH}{{\mathcal H}}

\newtheorem{theorem}{Theorem}[section]
\newtheorem{lemma}[theorem]{Lemma}

\theoremstyle{remark}
\newtheorem{remark}[theorem]{Remark}

\theoremstyle{definition}

\allowdisplaybreaks

\title{Rigidity of balls in the solid mean value property
for polyharmonic functions}
\date{\today}
\author{Nicola Abatangelo}
\address{Dipartimento di Matematica, Alma Mater Studiorum Università di Bologna, P.zza di Porta S. Donato 5, 40126 Bologna, Italy.}
\email{nicola.abatangelo@unibo.it}



\begin{document}

\begin{abstract}
We show that balls are the only open bounded domains for which the mean value formula 
for polyharmonic functions holds. 
We do so by adapting an argument of \"U. Kuran for harmonic functions.
Also, we provide a quantitative version of the same result.
\end{abstract}

\maketitle

\section{Introduction}

The mean value property is a tool of the utmost importance in the theory of harmonic functions,
as it yields maximum principles, the Liouville theorem, the Harnack inequality, 
and it ties the theory of PDEs with the potential theory of stochastic processes.

Here we are interested in the analogue property enjoyed by polyharmonic functions, \textit{i.e.},
functions solving
\[
\Delta^m u=0
\qquad\text{in some~$\Omega\subseteq\R^n$ open and bounded, where~$m\in\N$ with~$m\geq 2$.}
\]
Our main goal is to show that the corresponding \textit{solid} mean value property can only hold on balls,
as it has been shown for harmonic functions in~\cites{MR0140700,MR0177124,MR0320348}.

Mean value formulas for polyharmonic functions can be traced back to~\cite{03027684} and~\cite{03024989}. 
Later, similar ones were re-discovered in~\cite{MR0192074}. 
Below, we present the formulations of ~\cite{MR0192074}*{Table B} and~\cite{MR3176589}*{Theorem 3.1}.

\begin{lemma}\label{lem:mv}
Let~$D\subseteq\R^n$ be open and~$m\in\N$ with~$m\geq 2$. A function~$u\in L^1(D)\cap C^0(D)$ 
is~$m$-polyharmonic in~$D$ if and only if for any~$x_0\in D$,~$r>0$ such that~$B_r(x_0)\subseteq D$,
and~$0<\alpha_1<\cdots<\alpha_m\leq 1$ it holds 
\begin{align}\label{eq:coeff}
u(x_0) = \sum_{k=1}^m (-1)^{k+1} c_k\fint_{B_{\alpha_kr}(x_0)}u 
= \fint_{B_r}\Bigg(\sum_{k=1}^m (-1)^{k+1} c_k u(x_0+\alpha_k y) \Bigg)dy
\end{align}
where~$c_1,\ldots,c_m>0$ are constants depending on~$m$ and~$\alpha_1,\ldots,\alpha_m$
satisfying
\begin{align*}
\sum_{k=1}^m (-1)^{k+1} c_k = 1.
\end{align*}
\end{lemma}

To our purposes, in the notations of the preceding Lemma,
we also need to know the explicit dependence 
of the coefficients~$c_1,\ldots,c_m$ on~$\alpha_1,\ldots,\alpha_m$. 
For this reason, it is useful to recall it also to ensure their positivity 
(this is indeed not explicitly stated in the references).
By~\cite{MR0192074}*{Table B} and~\cite{MR3176589}*{equations (3.1) and (3.2)},
defining
\begin{align}\label{eq:cmatrix}
\R^{m\times m}\ni\ V:=\big[\alpha_i^{2j}\big]_{1\leq i\leq m,\ 0\leq j\leq m-1}
=\left[\begin{array}{cccc}
1 & \alpha_1^2 & \cdots & \alpha_1^{2(m-1)} \\
1 & \alpha_2^2 & \cdots & \alpha_2^{2(m-1)} \\
\vdots & \vdots & \ddots & \vdots \\
1 & \alpha_m^2 & \cdots & \alpha_m^{2(m-1)}
\end{array}\right]
\end{align}
and
\begin{align}\label{minors}
v_{k,1}\text{ as the minor obtained by removing from~$V$ the~$k$-th row and the first column},
\ k=1,\ldots,m,
\end{align}
one has 
\begin{align}\label{eq:coef}
c_k=\frac{v_{k,1}}{\det V}
\qquad\text{for any }k=1,\ldots,m.
\end{align}
The matrix~$V$ is usually called a \textit{Vandermonde matrix}.
The fact that~$0<\alpha_1<\ldots<\alpha_m$
makes it an example of totally positive matrix,
meaning a matrix whose minors are all positive
(see for example the monograph~\cite{MR2791531}*{Paragraph 3.0}):
for this reason, $v_{k,1}>0$ and $c_k>0$ for any $k=1,\ldots,m$.
The fact that 
\begin{align*}
\sum_{k=1}^m (-1)^{k+1} c_k = 
\frac1{\det V}\sum_{k=1}^m (-1)^{k+1} v_{k,1} = 1
\end{align*}
is a consequence of the Laplace's expansion for the determinant of~$V$.

In the particular case~$m=2$, \textit{i.e.}, in the particular case of the bilaplacian operator,
the following mean value formula can be found in~\cite{MR0547093}*{equation (4')}, 
although it remains a consequence of the one obtained in~\cite{03027684}.

\begin{lemma}
Let~$D\subseteq\R^n$ be open. A function~$u\in L^1(D)\cap C^0(D)$ 
is biharmonic in~$D$ if and only if for any~$x_0\in D$,~$r>0$ such that~$B_r(x_0)\subseteq D$,
and~$\alpha\in(0,1)$ it holds  
\begin{align*}
u(x_0) &= \frac1{1-\alpha^2}\fint_{B_{\alpha r}(x_0)}u-\frac{\alpha^2}{1-\alpha^2}\fint_{B_r(x_0)} u \\
&=\frac1{1-\alpha^2}\fint_{B_r}u(x_0+\alpha z)\;dz
-\frac{\alpha^2}{1-\alpha^2}\fint_{B_r}u(x_0+z)\;dz. \nonumber
\end{align*}
\end{lemma}

In order to state our main result, we need the following notation.
For~$\Omega\subseteq\R^n$ open,~$x_0\in\Omega$, and~$\lambda>0$ denote by
\begin{align*}
\Omega_\lambda(x_0):=\big\{x_0+\lambda(x-x_0):x\in\Omega\big\}.
\end{align*}
In this way,~$\Omega_\lambda(x_0)$ accounts for a rescaling of~$\Omega$
around its point~$x_0$.
It is in general not true that~$\Omega_\lambda(x_0)\subseteq\Omega$ whenever~$\lambda\in(0,1)$,
as this would give that~$\Omega$ is star-shaped around~$x_0$.
Nevertheless, for~$\Omega$ bounded and~$\lambda$ sufficiently small this will be true:
on the one hand---being~$\Omega$ open---one can find~$r>0$ such that~$B_r(x_0)\subseteq\Omega$,
and on the other hand---being~$\Omega$ bounded---$d_\Omega:=\diam(\Omega)<+\infty$;
then, for~$\lambda\in(0,r/d_\Omega]$, it will hold that~$\Omega_\lambda(x_0)\subseteq B_r(x_0)\subseteq\Omega$.

\begin{theorem}\label{thm:rigid-poly}
Let~$\Omega\subseteq\R^n$ be open and bounded,~$m\in\N$ with~$m\geq 2$,~$x_0\in\Omega$, and~$r>0$ such that~$B_r(x_0)\subseteq\Omega$.
If for any~$m$-polyharmonic function~$u\in L^1(\Omega)\cap C^0(\Omega)$,~$0<\alpha_1<\ldots<\alpha_{m-1}\leq r/d_\Omega$,
it holds
\begin{align}\label{mean-omega}
u(x_0) &= (-1)^{m+1} c_m\fint_{\Omega}u+\sum_{k=1}^{m-1} (-1)^{k+1} c_k\fint_{\Omega_{\alpha_k}(x_0)}u  \\
&= \fint_\Omega\Bigg((-1)^{m+1} c_mu(y)+\sum_{k=1}^{m-1}(-1)^{k+1} c_ku\big(x_0+\alpha_k(y-x_0)\big)\Bigg)dy \nonumber
\end{align}
where~$c_1,\ldots,c_m>0$ have been defined in \eqref{eq:coef},
then~$\Omega$ is a ball centered at~$x_0$.
\end{theorem}

\begin{remark}
Mind that \eqref{mean-omega} is the analogue of \eqref{eq:coeff}
using~$\Omega$ instead of~$B_r(x_0)$ as ``averaging domain''
and fixing~$\alpha_m=1$.
\end{remark}

Let us now introduce the notation
\[
\cH_m(\Omega)=\Big\{u\in L^1(\Omega)\cap C(\Omega):u\text{ is~$m$-polyharmonic in }\Omega\Big\},
\]
and, inspired by~\cite{MR4205791}*{equation (1.2)}, the higher-order Gauss mean value gap
\begin{align}
G_m(\Omega,x_0) :=
\sup\Bigg\{\frac1{M_\alpha(u)} 
\Bigg|u(x_0)-\fint_\Omega \Bigg((-1)^{m+1} c_mu(y)+\sum_{k=1}^{m-1}(-1)^{k+1} c_ku\big(x_0+\alpha_k(y-x_0)\big)\Bigg)dy\Bigg|: \nonumber \\
u\in\cH_m(\Omega), u\not\equiv 0,
\text{ and }\alpha_1,\ldots,\alpha_{m-1}\in\bigg(0,\frac{r}{d_\Omega}\bigg]
\Bigg\}, \label{Gm}
\end{align}
where
\begin{align}\label{Malpha}
M_\alpha(u) :=
\fint_\Omega\Bigg|(-1)^{m+1} c_m u(y)+\sum_{k=1}^{m-1} (-1)^{k+1} c_k \,
u\big(x_0+\alpha_k(y-x_0)\big)\Bigg|\;dy.
\end{align}
It is possible to give an alternative version of Theorem~\ref{thm:rigid-poly}
in quantitative form, by estimating the difference between the domain $\Omega$ 
and the ball $B_r(x_0)$ in terms of the above Gauss mean value gap.

\begin{theorem}\label{thm:stab-poly}
Let~$\Omega\subseteq\R^n$ be open and bounded,~$m\in\N$ with~$m\geq 2$, 
and~$x_0\in\Omega$ with~$r:=\dist(x_0,\partial\Omega)>0$.
There exists~$C>0$, depending only on~$n$ and~$m$, such that
\begin{align*}
\frac{|\Omega\setminus B_r(x_0)|}{|\Omega|}\leq C\frac{\diam(\Omega)^{m^2-m}}{r^{m^2-m}}\,G_m(\Omega,x_0).
\end{align*}
\end{theorem}

We give the proof of Theorem~\ref{thm:rigid-poly} in Section~\ref{sec:rig}
after a preliminary Lemma. Section~\ref{sec:stab} contains the proof of Theorem~\ref{thm:stab-poly}.
Finally, we defer to the Appendix some facts about the coefficient matrix $V$
defined in~\eqref{eq:cmatrix} that we need in our analysis.

\section{Rigidity of balls: proof of Theorem \texorpdfstring{\ref{thm:rigid-poly}}{rigid-poly}}
\label{sec:rig}

We follow a strategy designed in~\cite{MR0320348} for the analogue problem
in the harmonic case. For this, we first need to carefully pick a specific 
$m$-polyharmonic function.

\begin{lemma}
Let~$m\in\N, m\geq 2$.
For any~$z\in\R^n$ and~$\lambda_1,\ldots,\lambda_{m-1}\in\R$,  the function
\begin{align*}
\R^n\setminus\{z\}\ x\longmapsto\frac{|x|^2-|z|^2}{{|x-z|}^n}\:\prod_{k=1}^{m-1}\!\Big(\lambda_k^2-|x|^2\Big)
\end{align*}
is~$m$-polyharmonic in~$\R^n\setminus\{z\}$.
\end{lemma}
\begin{proof}
Let us consider the function
\begin{align*}
h(x):=z\cdot x\:\prod_{k=1}^{m-1}\!\Big(
\big(\lambda_k^2-|z|^2\big)|x|^2
-2\big(\lambda_k^2+|z|^2\big)z\cdot x
+\big(\lambda_k^2-|z|^2\big)|z|^2
\Big)
\qquad\text{for any }x\in\R^n.
\end{align*}
As~$h$ is a polynomial of degree~$2m-1$, it is~$m$-polyharmonic in~$\R^n$.
Also, we may as well write it as
\begin{align*}
h(x)=z\cdot x\:\prod_{k=1}^{m-1}\!\Big(
\big(\lambda_k^2-|z|^2\big)|x-z|^2
-4|z|^2z\cdot x
\Big)
\qquad\text{for any }x\in\R^n.
\end{align*}
Its Kelvin transform, with pole at~$z$, is given by
\begin{align*}
Kh(x) &= {|x-z|}^{2m-n}h\bigg(2|z|^2\frac{x-z}{{|x-z|}^2}+z\bigg) \\
&=
{|x-z|}^{2m-n}\bigg(2|z|^2\frac{z\cdot x-|z|^2}{{|x-z|}^2}+|z|^2\bigg)\times \\
&\qquad\times
\prod_{k=1}^{m-1}\Bigg[
\big(\lambda_k^2-|z|^2\big)\frac{4|z|^4}{{|x-z|}^2}
-4|z|^2\bigg(2|z|^2\frac{z\cdot x-|z|^2}{{|x-z|}^2}+|z|^2\bigg)
\Bigg] \\
&=
\frac{1}{{|x-z|}^n}
\bigg(2|z|^2z\cdot x-2|z|^4+|z|^2|x-z|^2\bigg)\times \\
&\qquad\times
\prod_{k=1}^{m-1}\Bigg[
\big(\lambda_k^2-|z|^2\big)4|z|^4
-4|z|^2\bigg(2|z|^2z\cdot x-2|z|^4+|z|^2|x-z|^2\bigg)
\Bigg] \\
&=
\frac{4^{m-1}|z|^{2+4(m-1)}}{{|x-z|}^n}
\big(|x|^2-|z|^2\big)
\prod_{k=1}^{m-1}\Bigg[
\lambda_k^2-|z|^2
-2z\cdot x+2|z|^2-|x-z|^2
\Bigg] \\
&=
\frac{4^{m-1}|z|^{2+4(m-1)}}{{|x-z|}^n}
\big(|x|^2-|z|^2\big)
\prod_{k=1}^{m-1}\Big(\lambda_k^2-|x|^2\Big)
\qquad\text{for any }x\in\R^n\setminus\{z\}.
\end{align*}
As the Kelvin transform preserves the~$m$-polyharmonicity
(see~\cite{MR2667016}*{Lemma 6.14}) 
our claim is proved.
\end{proof}

\begin{proof}[Proof of Theorem \ref{thm:rigid-poly}]
Let~$r:=\dist(x_0,\partial\Omega)>0$ so that~$B_r(x_0)\subseteq\Omega$
and it is possible to choose~$z\in\partial B_r(x_0)\cap \partial\Omega$.
For simplicity, and without loss of generality, suppose that~$x_0=0$.

Suppose we are given an~$m$-polyharmonic function~$u$ in~$\Omega$ such that~$u(0)=0$ and
\begin{align*}
0 = (-1)^m c_m\fint_\Omega u+\sum_{k=1}^{m-1} (-1)^k c_k\fint_{\Omega_{\alpha_k}}u.
\end{align*}
Then
\begin{align*}
0 = (-1)^m c_m\int_\Omega u+\sum_{k=1}^{m-1} (-1)^k \frac{c_k}{\alpha_k^n}\int_{\Omega_{\alpha_k}}u.
\end{align*}
The analogous mean value formula holds on balls (\textit{cf}. \eqref{eq:coeff} with~$\alpha_m=1$), 
so subtracting the two we obtain
\begin{align}\label{eq:zero-sum}
0 = (-1)^m c_m\int_{\Omega\setminus B_r}u + \sum_{k=1}^{m-1}  (-1)^k \frac{c_k}{\alpha_k^n}\int_{\Omega_{\alpha_k}\setminus B_{\alpha_k r}}u.
\end{align}

We now choose
\begin{align}\label{alphas}
\alpha_k:=\bigg(\frac{r}{2d_\Omega}\bigg)^{m-k}
\qquad\text{for any }k=1,\ldots,m-1.
\end{align}
This choice entails in particular that (since~$2r\leq d_\Omega$)
\begin{align*}
\alpha_k=\frac{r}{2d_\Omega}\,\alpha_{k+1}<\alpha_{k+1}
\qquad \text{for any }k=1,\ldots,m-2,
\end{align*}
and therefore
\begin{align}\label{eq:inclusions}
B_{\alpha_k r}\subseteq\Omega_{\alpha_k} \subseteq B_{\alpha_{k+1}r} \qquad \text{for any }k=1,\ldots,m-1.
\end{align}

Define 
\begin{align}\label{u-poly-case}
u(x):=\frac{|x|^2\big(r^2-|x|^2\big)}{{|x-z|}^n}\:
\prod_{k=2}^{m-1}\!\Big(\alpha_k^2r^2-|x|^2\Big)
\qquad\text{for any }x\in\R^n\setminus\{z\}.
\end{align}
By the preceding Lemma,~$u$ is~$m$-polyharmonic in~$\R^n\setminus\{z\}$
and \textit{a fortiori} in~$\Omega$. Note that~$u(0)=0$.
Also,~$u$ alternates sign on the balls~$B_{\alpha_kr}$, meaning that
\begin{align*}
u\geq 0\text{ in }B_{\alpha_2r},\quad
u\leq 0\text{ in }B_{\alpha_3r}\setminus B_{\alpha_2r},\quad
u\geq 0\text{ in }B_{\alpha_4r}\setminus B_{\alpha_3r},\quad
\textit{etc...}
\end{align*}
This means that
\begin{align*}
(-1)^{m+1}u\chi_{\R^n\setminus B_r}\geq 0,
\qquad
(-1)^{k+1} u\chi_{B_{\alpha_{k+1}r}\setminus B_{\alpha_kr}}\geq 0
\quad\text{for any }k=1,\ldots,m-1.
\end{align*}
By \eqref{eq:inclusions}, it also holds
\begin{align}\label{constant-sign}
(-1)^{m+1}u\chi_{\Omega\setminus B_r}\geq 0,
\qquad
(-1)^{k+1}u\chi_{\Omega_{\alpha_k}\setminus B_{\alpha_kr}}\geq 0
\qquad\text{for any }k=1,\ldots,m-1,
\end{align}
which is incompatible with \eqref{eq:zero-sum}, unless~$|\Omega\setminus B_r|=0$.
\end{proof}

\section{Stability via the Kuran gap: proof of Theorem \texorpdfstring{\ref{thm:stab-poly}}{stab-poly}}
\label{sec:stab}

\begin{lemma}
Let~$\Omega\subseteq\R^n$ be open and bounded,~$m\in\N$ with~$m\geq 2$, 
and~$x_0\in\Omega$. Then~$G_m(\Omega,x_0)<+\infty$.
\end{lemma}
\begin{proof}
Owing to Lemma \ref{lem:mv}, for any~$u\in\cH_m(\Omega)$ it holds 
\[
u(x_0) 
= \fint_{B_r}\Bigg(\sum_{k=1}^m (-1)^{k+1} c_k u(x_0+\alpha_k y) \Bigg)dy
\qquad\text{for any }r>0\text{ such that }B_r(x_0)\subseteq\Omega.
\]
Therefore, recalling \eqref{Gm} and \eqref{Malpha},
\begin{align*}
& \Bigg|u(x_0)-\fint_\Omega \Bigg((-1)^{m+1} c_mu(y)+\sum_{k=1}^{m-1}(-1)^{k+1} c_ku\big(x_0+\alpha_k(y-x_0)\big)\Bigg)dy\Bigg| \leq \\
& \leq 
|u(x_0)|+M_\alpha(u) \\
& \leq
\fint_{B_r}\Bigg|\sum_{k=1}^m (-1)^{k+1} c_k u(x_0+\alpha_k y) \Bigg|dy
+M_\alpha(u) \\
& \leq 
\frac{|\Omega|}{|B_r|}M_\alpha(u)+M_\alpha(u).
\end{align*}
Thus
\[
G_m(\Omega,x_0)\leq \frac{|\Omega|}{|B_r|}+1,
\]
which proves our claim.
\end{proof}

\begin{proof}[Proof of Theorem \ref{thm:stab-poly}]
Suppose without loss of generality that~$x_0=0$ and that~$r:=\dist(0,\partial\Omega)$.
Pick~$z\in\partial B_r\cap\partial\Omega$. 
Let us consider the function defined in \eqref{u-poly-case}
(together with \eqref{alphas}), 
which we recall being~$m$-polyharmonic in~$\R^n\setminus\{z\}$ and, therefore, in~$\Omega$.
Then
\begin{align*}
G_m(\Omega,0) & \geq
\frac1{M_\alpha(u)}
\Bigg|u(0)-(-1)^{m+1}c_m\fint_\Omega u-\sum_{k=1}^{m-1} (-1)^{k+1} c_k\fint_{\Omega_{\alpha_k}}u\Bigg| \\
& =
\frac1{|\Omega|}
\Bigg|
(-1)^{m+1}c_m\int_\Omega u
+\sum_{k=1}^{m-1} (-1)^{k+1} \frac{c_k}{\alpha_k^n}\int_{\Omega_{\alpha_k}}u\Bigg|.
\end{align*}
Note that in virtue of Lemma \ref{lem:mv} we have
\[
(-1)^{m+1}c_m\int_{B_r} u
+\sum_{k=1}^{m-1} (-1)^{k+1} \frac{c_k}{\alpha_k^n}\int_{B_{\alpha_kr}}u
=|B_r|u(0)=0
\]
and therefore
\begin{align*}
(-1)^{m+1}c_m\int_\Omega u
+\sum_{k=1}^{m-1} (-1)^{k+1} \frac{c_k}{\alpha_k^n}\int_{\Omega_{\alpha_k}}u
=
(-1)^{m+1}c_m\int_{\Omega\setminus B_r} u
+\sum_{k=1}^{m-1} (-1)^{k+1} \frac{c_k}{\alpha_k^n}\int_{\Omega_{\alpha_k}\setminus B_{\alpha_kr}}u.
\end{align*}
So,
\begin{align}\label{pause-poly}
G_m(\Omega,0) \geq
\frac1{|\Omega|M_\alpha(u)}\Bigg|
(-1)^{m+1}c_m\int_{\Omega\setminus B_r} u
+\sum_{k=1}^{m-1} (-1)^{k+1} \frac{c_k}{\alpha_k^n}\int_{\Omega_{\alpha_k}\setminus B_{\alpha_kr}}u\Bigg|.
\end{align}
We have already noticed, see \eqref{constant-sign}, 
that~$(-1)^{k+1}u\chi_{\Omega_{\alpha_k}\setminus B_{\alpha_kr}}\geq 0$ 
for any~$k=1,\ldots,m$.
This allows us to drop the absolute value in \eqref{pause-poly} and write
\begin{align*}
G_m(\Omega,0) &\geq 
\frac1{|\Omega|M_\alpha(u)}
\Bigg[
(-1)^{m+1}c_m\int_{\Omega\setminus B_r} u
+\sum_{k=1}^{m-1} (-1)^{k+1} \frac{c_k}{\alpha_k^n}\int_{\Omega_{\alpha_k}\setminus B_{\alpha_kr}}u\Bigg] \\
& =
\frac1{|\Omega|M_\alpha(u)}
\Bigg[(-1)^{m+1}c_m\int_{\Omega\setminus B_r} u +
\sum_{k=1}^{m-1} (-1)^{k+1} c_k\int_{\Omega\setminus B_r}u(\alpha_kx)\;dx \Bigg] \\
& \geq 
\frac{c_1}{|\Omega|M_\alpha(u)}
\int_{\Omega\setminus B_r}u(\alpha_1x)\;dx.
\end{align*}
In this setting, we have
\begin{align*}
\int_{\Omega\setminus B_r}u(\alpha_1x)\;dx\geq \frac{\alpha_1^2r^2\big(r^2-\alpha_1^2d_\Omega^2\big)}{(\alpha_1d_\Omega+r)^n}
|\Omega\setminus B_r|\prod_{k=2}^{m-1}\Big(\alpha_k^2r^2-\alpha_1^2d_\Omega^2\Big)
\end{align*}
and, using the choice done in \eqref{alphas}, 
\begin{align*}
\int_{\Omega\setminus B_r}u(\alpha_1x)\;dx
& \geq 
\frac{r^{2m}}{2^{2m-2}d_\Omega^{2m-2}}
\bigg(\frac{r^{m-1}}{2^{m-1}d_\Omega^{m-2}}+r\bigg)^{-n}|\Omega\setminus B_r|
\prod_{k=2}^m\bigg(\frac{r^{2m-2k+2}}{2^{2m-2k}d_\Omega^{2m-2k}}
-\frac{r^{2m-2}}{2^{2m-2}d_\Omega^{2m-4}}\bigg) \\
& =
\frac{r^{2m-n}}{2^{2m-2}d_\Omega^{2m-2}}
\bigg(1+\frac{r^{m-2}}{2^{m-1}d_\Omega^{m-2}}\bigg)^{-n}|\Omega\setminus B_r|
\prod_{k=2}^m\frac{r^{2m-2k+2}}{2^{2m-2k}d_\Omega^{2m-2k}}\bigg(1
-\frac{r^{2k-4}}{2^{2k-2}d_\Omega^{2k-4}}\bigg) \\
& \geq
C_{m,n}\frac{r^{2m-n}}{d_\Omega^{2m-2}}|\Omega\setminus B_r|
\prod_{k=2}^m\frac{r^{2m-2k+2}}{d_\Omega^{2m-2k}} \\
& =
C_{m,n}\frac{r^{2m-n}}{d_\Omega^{2m-2}}
\frac{r^{2m^2-2-2\sum_{k=2}^mk}}{d_\Omega^{2m^2-2m-2\sum_{k=2}^mk}}
|\Omega\setminus B_r| \\
& =
C_{m,n}
\frac{r^{m^2+m-n}}{d_\Omega^{m^2-m}}
|\Omega\setminus B_r|.
\end{align*}
We then deduce
\begin{align*}
C_{m,n}
c_1|\Omega\setminus B_r|
\leq
\frac{d_\Omega^{m^2-m}}{r^{m^2+m-n}}
M_\alpha(u)|\Omega| G_m(\Omega,0).
\end{align*}

Finally, we estimate (recall \eqref{Malpha})
\begin{align*}
|\Omega|\,M_\alpha(u)
=
\int_\Omega\Bigg|(-1)^{m+1} c_m u(y)+\sum_{k=1}^{m-1} (-1)^{k+1} c_k \,
u\big(\alpha_ky)\big)\Bigg|\;dy 
\leq 
c_m\int_\Omega|u| +
\sum_{k=1}^{m-1} c_k \int_\Omega\big|u(\alpha_ky)\big|\;dy 
\end{align*}
where, for any~$k=1,\ldots,m-1$,
\begin{align*}
\int_\Omega\big|u(\alpha_ky)\big|\;dy 
&=
\int_\Omega\frac{\alpha_k^2|y|^2}{{|\alpha_k y-z|}^n}\:
\prod_{j=2}^m\!\Big|\alpha_j^2r^2-\alpha_k^2|y|^2\Big|\;dy \\
&=
\int_\Omega\frac{\alpha_k^2|y|^2}{{|\alpha_k y-z|}^n}\:
\bigg[\prod_{j=k+1}^m\!\Big(\alpha_j^2r^2-\alpha_k^2|y|^2\Big]\bigg)
\bigg[\prod_{j=2}^k\!\Big|\alpha_j^2r^2-\alpha_k^2|y|^2\Big|\bigg]\;dy \\
& \leq 
\int_\Omega\frac{\alpha_k^2|y|^2}{{|\alpha_k y-z|}^n}\:
\bigg[\prod_{j=k+1}^m\!\Big(\alpha_j^2r^2\Big)\bigg]
\bigg[\prod_{j=2}^k\!\Big(\alpha_j^2r^2+\alpha_k^2d_\Omega^2\Big)\bigg]\;dy \\
& \leq 
|\Omega|\frac{\alpha_k^2d_\Omega^2}{\big(r-\alpha_kd_\Omega\big)^n}
\bigg[\prod_{j=k+1}^m\!\Big(\alpha_j^2r^2\Big)\bigg]
\bigg[\prod_{j=2}^k\!\Big(\alpha_j^2r^2+\alpha_k^2d_\Omega^2\Big)\bigg] \\
& \leq 
\frac{C_n|\Omega|}{r^n} \bigg(\frac{r}{d_\Omega}\bigg)^{2m-2k}d_\Omega^2r^{2m-2k}\bigg(\frac{r}{d_\Omega}\bigg)^{2\sum_{j=k+1}^m(m-j)}
d_\Omega^{2k-2} \\
& =
\frac{C_n|\Omega|}{r^n} \bigg(\frac{r}{d_\Omega}\bigg)^{2m-2k}d_\Omega^{2k}r^{2m-2k}\bigg(\frac{r}{d_\Omega}\bigg)^{m^2-m-2mk+k^2+k}
\\
& \leq 
\frac{C_n|\Omega|}{r^n} \bigg(\frac{r}{d_\Omega}\bigg)^{2m-2k}d_\Omega^{2k}r^{2m-2k}\bigg(\frac{r}{d_\Omega}\bigg)^{m-k-1} \\
& =
C_n|\Omega| d_\Omega^{5k-3m+1}r^{5m-5k-1-n}.
\end{align*}
In the last passage above, we have used that
\[
m^2-m-2mk+k^2+k\geq m-k-1
\qquad\text{for any }k=1,\ldots,m-1.
\]
Using the above estimate, we are now in position to state that
\begin{align*}
|\Omega|\,M_\alpha(u)
\leq c_m \int_\Omega|u| +
\sum_{k=1}^m c_k \int_\Omega\big|u(\alpha_ky)\big|\;dy 
\leq c_m \int_\Omega|u| +
C_n|\Omega| \sum_{k=1}^{m-1} c_kd_\Omega^{5k-3m+1}r^{5m-5k-1-n}.
\end{align*}
Given that
\begin{align*}
\int_\Omega\big|u(y)\big|\;dy 
&=
\int_\Omega\frac{|y|^2\,\big|r^2-|y|^2\big|}{{|y-z|}^n}\:
\prod_{j=2}^{m-1}\!\Big|\alpha_j^2r^2-|y|^2\Big|\;dy \\
& \leq 
d_\Omega^2
\int_\Omega\frac{\big|r^2-|y|^2\big|}{{|y-z|}^n}\;dy 
\prod_{j=2}^{m-1}\!\Big(\alpha_j^2r^2+d_\Omega^2\Big) \\
& \leq 
d_\Omega^2
\int_\Omega\frac{\big|r^2-|y|^2\big|}{{|y-z|}^n}\;dy 
\prod_{j=2}^{m-1}\!\bigg(4d_\Omega^2+d_\Omega^2\bigg) \\
& \leq 
5^{m-2}d_\Omega^{2m-2}
\int_\Omega\frac{\big|r^2-|y|^2\big|}{{|y-z|}^n}\;dy \\
& \leq 
5^{m-2}d_\Omega^{2m-2} C(n,|\Omega|)
\end{align*}
for some~$C>0$ depending only on the dimension~$n$ and the measure of~$\Omega$
(this is justified in~\cite{MR4205791}*{Lemma 2.1}),
we obtain
\begin{align*}
|\Omega|\,M_\alpha(u)
\leq 
C(n,m,|\Omega|)\Bigg(c_md_\Omega^{2m-2}+
\sum_{k=1}^{m-1} c_kd_\Omega^{5k-3m+1}r^{5m-5k-1-n}\Bigg).
\end{align*}
The coefficients~$c_k$'s, defined in \eqref{eq:coef} and
under the particular choice in \eqref{alphas}, have been computed in
Lemma~\ref{lem:coef}: accordingly, we estimate
\[
c_k\leq 2^{m-1}\bigg(\frac{r}{d_\Omega}\bigg)^{k^2-k}
\qquad\text{for any }k=1,\ldots,m.
\]
This yields
\begin{align*}
|\Omega\setminus B_r|
&\leq 
C(n,m,|\Omega|)\,G_m(\Omega,0)\,\frac{d_\Omega^{m^2-m}}{r^{m^2+m-n}}
\Bigg(r^{m^2-m}d_\Omega^{-m^2+3m-2}+\sum_{k=1}^{m-1}d_\Omega^{-k^2+6k-3m+1}r^{5m+k^2-6k-1-n}\Bigg) \\
&=
C(n,m,|\Omega|)\,G_m(\Omega,0)\,
\Bigg(r^{n-2m}d_\Omega^{2m-2}+
\frac{d_\Omega^{m^2-m}}{r^{m^2-m}}\sum_{k=1}^{m-1}\bigg(\frac{r}{d_\Omega}\bigg)^{3m+k^2-6k-1}\Bigg).
\end{align*}
Using now that\footnote{
For~$m\geq 4$, the term~$k^2-6k$ is minimized at~$k=3$ where it equals~$-9$ and~$3m-9-1\geq 2$;
for~$m=3$, the claimed inequality becomes~$k^2-6k+8=(4-k)(2-k)\geq 0$ which is true for~$k=1,2$; 
for~$m=2$ and~$k=1$, it is trivially verified.
}
\[
r<d_\Omega
\qquad\text{and}\qquad
3m+k^2-6k-1\geq 0
\quad\text{for any }k=1,\ldots,m-1,
\]
we can simply estimate
\begin{align*}
|\Omega\setminus B_r| \leq 
C(n,m,|\Omega|)\,G_m(\Omega,0)\,
\Bigg(r^{n-2m}d_\Omega^{2m-2}+
\frac{d_\Omega^{m^2-m}}{r^{m^2-m}}\Bigg).
\end{align*}
As~$G_m(\Omega,0)$ is scaling invariant, by rescaling we get
\[
\frac{|\Omega\setminus B_r|}{|\Omega|} \leq 
C(n,m,1)\,G_m(\Omega,0)\,
\Bigg(\frac{r^{n-2m}d_\Omega^{2m-2}}{|\Omega|^{1-\frac2n}}+
\frac{d_\Omega^{m^2-m}}{r^{m^2-m}}\Bigg)
\]
and using that
\begin{align*}
|B_1|r^n=|B_r|\leq|\Omega|
\quad\Rightarrow\quad
r\leq\bigg(\frac{|\Omega|}{|B_1|}\bigg)^\frac1n
\end{align*}
we are left with
\begin{align*}
\frac{|\Omega\setminus B_r|}{|\Omega|} 
&\leq 
C(n,m,1)\,G_m(\Omega,0)\,
\Bigg(\frac{d_\Omega^{2m-2}}{r^{2m-2}}+
\frac{d_\Omega^{m^2-m}}{r^{m^2-m}}\Bigg) \\
&\leq 
C(n,m,1)\,G_m(\Omega,0)\,
\frac{d_\Omega^{m^2-m}}{r^{m^2-m}}.
\end{align*}
\end{proof}

\appendix
\section{The coefficient matrix}\label{app}

We have introduced in \eqref{eq:cmatrix} the matrix~$V$ of the coefficients needed in the 
mean value formula \eqref{eq:coeff}. We collect here the properties that we need in our analysis.
Most of these are surely well known to the expert reader,
but since they might fall outside the cultural background of many others,
we provide some quick justifications.

First, let us compute its determinant.

\begin{lemma}
Let~$V$ be the matrix in \eqref{eq:cmatrix}. Then
\[
\det(V)=\prod_{1\leq i<j\leq m}\Big(\alpha_j^2-\alpha_i^2\Big).
\]
\end{lemma}
\begin{proof}
Let us proceed by induction. Let~$V_k$ be the~$(m-1)\times(m-1)$ matrix obtained by
removing from~$V$ the~$m$-th row and~$k$-th column, for~$k=1,\ldots,m$. 
By using the Laplace's expansion for the determinant, we write
\[
\det(V)=\sum_{k=1}^m (-1)^{k+m}\alpha_m^{2(k-1)}\det(V_k).
\]
In this setting,~$\det V$ is a polynomial of degree (at most)~$m-1$ in the variable~$\alpha_m^2$.
This polynomial has roots in~$\alpha_1^2,\ldots,\alpha_{m-1}^2$ and so
\[
\det(V)=\det(V_m)\prod_{i=1}^{m-1}\Big(\alpha_m^2-\alpha_i^2\Big).
\]
As the submatrix~$V_m$ is of Vandermonde type, we use the induction hypothesis to say
\[
\det(V_m)=\prod_{1\leq i<j\leq m-1}\Big(\alpha_j^2-\alpha_i^2\Big)
\]
and deduce the thesis.
\end{proof}

When~$V$ has entries like the ones defined in \eqref{alphas}, we also need the minors
listed in \eqref{minors}. This calculation is performed in the next Lemma.

\begin{lemma}\label{lem:coef}
For~$x\in(0,1)$, let 
\[
\overline V:=\Big[x^{2j(m-i)}\Big]_{1\leq i\leq m,\,0\leq j\leq m-1}.
\]
If~$\overline v_{k,1}$ is the minor obtained by removing from~$\overline V$ the~$k$-th row 
and the first column, then
\[
\overline v_{k,1}=\frac{x^{k^2-k}}
{\prod_{i=1}^{k-1}\Big(1-x^{2(k-i)}\Big)\prod_{j=k+1}^m\Big(1-x^{2(j-k)}\Big)}\,\det(\overline V).
\]
In particular,
\[
0<\frac{\overline v_{k,1}}{\det(\overline V)}<\frac{x^{k^2-k}}{\big(1-x^2\big)^{m-1}}.
\]
\end{lemma}
\begin{proof}
We have
\begin{align*}
\overline v_{k,1}
&=
\det\Big[x^{2j(m-i)}\Big]_{1\leq i\leq m,\,i\neq k,\,1\leq j\leq m-1} \\
&=
\det\Big[x^{2(m-i)}x^{2(j-1)(m-i)}\Big]_{1\leq i\leq m,\,i\neq k,\,1\leq j\leq m-1} \\
&=
\left[\prod_{\substack{i=1 \\ i\neq k}}^mx^{2(m-i)}\right]\det\Big[x^{2(j-1)(m-i)}\Big]_{1\leq i\leq m,\,i\neq k,\,1\leq j\leq m-1} \\
&=
x^{m^2-3m+2k}\det\Big[x^{2(j-1)(m-i)}\Big]_{1\leq i\leq m,\,i\neq k,\,1\leq j\leq m-1}.
\end{align*}
The determinant we are left to compute is the one of a Vandermonde matrix, 
so by the previous Lemma we know that
\begin{align*}
&\det\Big[x^{2(j-1)(m-i)}\Big]_{1\leq i\leq m,\,i\neq k,\,1\leq j\leq m-1}
=\\
&\quad=\prod_{\substack{1\leq i<j\leq m \\ i,j\neq k}}\Big(x^{2(m-j)}-x^{2(m-i)}\Big) \\
&\quad=\det(\overline V)
\left[\prod_{i=1}^{k-1}\Big(x^{2(m-k)}-x^{2(m-i)}\Big)
\prod_{j=k+1}^m\Big(x^{2(m-j)}-x^{2(m-k)}\Big)
\right]^{-1} \\
&\quad=\det(\overline V)
\left[x^{2(m-k)(k-1)+(m-k)(m+k-3)}\prod_{i=1}^{k-1}\Big(1-x^{2(k-i)}\Big)
\prod_{j=k+1}^m\Big(1-x^{2(j-k)}\Big)
\right]^{-1}.
\end{align*}
We then deduce 
\[
\frac{\overline v_{k,1}}{\det(\overline V)}=\frac{x^{k^2-k}}
{\prod_{i=1}^{k-1}\Big(1-x^{2(k-i)}\Big)\prod_{j=k+1}^m\Big(1-x^{2(j-k)}\Big)}.
\]
\end{proof}

\section*{Acknowledgments}
The research of the author has been partially supported by the GNAMPA-INdAM project ``Equazioni nonlocali di tipo misto e geometrico'' (Italy), CUP E53C22001930001, and partially by the PRIN project 2022R537CS ``$NO^3$ - NOdal Optimization, NOnlinear elliptic equations, NOnlocal geometric problems, with a focus on regularity'' (Italy).
The author is also grateful to G. Cupini and E. Lanconelli for bringing this problem to his attention.


\begin{bibdiv}
\begin{biblist}

\bib{MR0192074}{article}{
   author={Bramble, J. H.},
   author={Payne, L. E.},
   title={Mean value theorems for polyharmonic functions},
   journal={Amer. Math. Monthly},
   volume={73},
   date={1966},
   number={4},
   pages={124--127},
}

\bib{MR3176589}{article}{
   author={Caramuta, P.},
   author={Cialdea, A.},
   title={Mean value theorems for polyharmonic functions: a conjecture by
   Picone},
   journal={Analysis (Berlin)},
   volume={34},
   date={2014},
   number={1},
   pages={51--66},
}

\bib{MR4205791}{article}{
   author={Cupini, Giovanni},
   author={Lanconelli, Ermanno},
   author={Fusco, Nicola},
   author={Zhong, Xiao},
   title={A sharp stability result for the Gauss mean value formula},
   journal={J. Anal. Math.},
   volume={142},
   date={2020},
   number={2},
   pages={587--603},
}

\bib{MR0140700}{article}{
   author={Epstein, Bernard},
   title={On the mean-value property of harmonic functions},
   journal={Proc. Amer. Math. Soc.},
   volume={13},
   date={1962},
   pages={830},
}

\bib{MR0177124}{article}{
   author={Epstein, Bernard},
   author={Schiffer, M. M.},
   title={On the mean-value property of harmonic functions},
   journal={J. Analyse Math.},
   volume={14},
   date={1965},
   pages={109--111},
}

\bib{MR2791531}{book}{
   author={Fallat, Shaun M.},
   author={Johnson, Charles R.},
   title={Totally nonnegative matrices},
   series={Princeton Series in Applied Mathematics},
   publisher={Princeton University Press, Princeton, NJ},
   date={2011},
   pages={xvi+248},
}

\bib{MR0547093}{article}{
   author={Fichera, Gaetano},
   title={Mean value theorems and majorization formulas for biharmonic
   functions},
   language={Italian},
   journal={Rend. Sem. Mat. Univ. Padova},
   volume={59},
   date={1978},
   pages={285--294 (1979)},
}

\bib{MR1298400}{article}{
   author={Fichera, Gaetano},
   title={Mean value theorems and characterization of the sphere in~${\bf
   R}^n$},
   journal={Rend. Sem. Mat. Messina Ser. II},
   volume={1(14)},
   date={1991},
   pages={91--103},
}

\bib{MR2667016}{book}{
   author={Gazzola, Filippo},
   author={Grunau, Hans-Christoph},
   author={Sweers, Guido},
   title={Polyharmonic boundary value problems},
   series={Lecture Notes in Mathematics},
   volume={1991},
   publisher={Springer-Verlag, Berlin},
   date={2010},
   pages={xviii+423},
}

\bib{MR0320348}{article}{
   author={Kuran, \"U.},
   title={On the mean-value property of harmonic functions},
   journal={Bull. London Math. Soc.},
   volume={4},
   date={1972},
   pages={311--312},
}

\bib{03024989}{book}{
 author={Nicolesco, Miron},
 book={
 title={Les fonctions polyharmoniques. Expos\'es sur la th\'eorie des fonctions. IV. Publi\'e par Paul Montel.},
 },
 language={French},
 title={Les fonctions polyharmoniques. Expos{\'e}s sur la th{\'e}orie des fonctions. IV. Publi{\'e} par Paul Montel.},
 date={1936},
}

\bib{03027684}{article}{
 author={Picone, Mauro},
 language={Italian},
 title={Sulla convergenza delle successioni di funzioni iperarmoniche},
 journal={Bulletin Math{\'e}matique de la Soci{\'e}t{\'e} Roumaine des Sciences},
 volume={38},
 number={2},
 pages={105--112},
 date={1936},
}

\end{biblist}
\end{bibdiv}
\end{document}